\documentclass{amsart}

\usepackage{amsthm}
\usepackage{amsmath,amssymb}
\usepackage{dsfont}
\usepackage{amssymb}
\usepackage{enumerate}
\usepackage{graphicx}
\usepackage{float}
\usepackage{bbm}
\usepackage{comment}
\usepackage{hyperref}

\hypersetup{
	colorlinks,
	citecolor=blue,
	filecolor=blue,
	linkcolor=blue,
	urlcolor=blue
}
\newtheorem{theorem}{Theorem}[section]
\newtheorem{proposition}[theorem]{Proposition}
\newtheorem{lemma}[theorem]{Lemma}
\newtheorem{problem}{Problem}

\newtheorem{corollary}[theorem]{Corollary}

\newtheorem{remark}[theorem]{Remark}

\theoremstyle{definition}
\newtheorem{definition}{Definition}

\newcommand{\TT}{{\mathbb T}}

\newcommand{\R}{\mathbb{R}}
\newcommand{\Z}{\mathbb{Z}}
\newcommand{\N}{\mathbb{N}}

\newcommand{\V}{\mathcal{V}}
\newcommand{\W}{\mathcal{W}}

\newcommand{\dsum}{
\displaystyle\sum}
\newcommand{\dint}{
\displaystyle\int}
\newcommand{\la}{\langle}
\newcommand{\ra}{\rangle}

\newcommand{\bracket}[1]{\langle #1 \rangle}

\newcommand{\Rev}{R\'{e}v\'{e}sz}

\def\Prob{\textnormal{Prob}\,}
\def\supp{{\textnormal{supp}\,}}

\def\sinc{{\textnormal{sinc}\,}}

\newcommand{\HH}{\mathcal{H}}

\newcommand{\eps}{\varepsilon}
\newcommand{\Epsilon}{\mathcal{E}}
\newcommand{\finsum}[3]{\underset{#1=#2}{\overset{#3}\sum}}

\title[Convergence of Rearranged Fourier Series]{An Operator theoretic approach to the convergence of rearranged Fourier series
}
\author{Keaton Hamm, Ben Hayes, and Armenak Petrosyan}
\begin{document}
		
\address{\textrm{(Keaton Hamm)}
Department of Mathematics,
University of Arizona,
Tucson, Arizona 85721 USA}
\email{hamm@math.arizona.edu}

\address{\textrm{(Ben Hayes)}
	Department of Mathematics,
	University of Virginia,
	Charlottesville, Virginia 22904 USA}
\email{brh5c@virginia.edu}

\address{\textrm{(Armenak Petrosyan)}
Computational and Applied Mathematics Group,
Oak Ridge National Laboratory,
 Oak Ridge, Tennessee 37830 USA}
\email{petrosyana@ornl.gov}


\begin{abstract}

This article studies the rearrangement problem for Fourier series introduced by P.L. Ulyanov, who posed the question if every continuous function on the torus admits a rearrangement of its Fourier coefficients such that the rearranged partial sums of the Fourier series converge uniformly to the function.  The main theorem here gives several new equivalences to this problem in terms of the convergence of the rearranged Fourier series in the strong (equivalently in this case, weak) operator topologies on $\mathcal{B}(L_2(\TT))$.  Additionally, a new framework for further investigation is introduced by considering convergence for subspaces of $L_2$, which leads to many methods for attempting to prove or disprove Ulyanov's problem.  In this framework, we provide characterizations of unconditional convergence of the Fourier series in the SOT and WOT.  These considerations also give rise to some interesting questions regarding weaker versions of the rearrangement problem.

Along the way, we consider some interesting questions related to the classical theory of trigonometric polynomials.  All of the results here admit natural extensions to arbitrary dimensions.

\end{abstract}

\maketitle

\section{Introduction}

In his {\em M\'{e}moire sur la propogation de la chaleur dans les corps solides} \cite{Fourier}, Fourier considered the expansion of functions in terms of trigonometric series to solve the heat equation.  His ideas later exploded into use, and have since found a fundamental place in many mathematical and other scientific disciplines. While Fourier's initial works are not rigorous by today's standards, the intervening 200 years has seen many fascinating results come out of an attempt to provide rigor to the sense in which a function may be represented by its Fourier series.  Indeed the answer to this question contains a remarkable amount of subtlety and is fraught with danger even for the seasoned veteran.

To give a brief overview, recall that the Riesz--Fischer Theorem implies that the Fourier series for $f\in L_2$ converges to $f$ in $L_2$.  The extension to $L_p$ for $1<p<\infty$ is due to M. Riesz \cite{RieszLp}.  The endpoints do not share this property, i.e. there are functions in $L_1$ and $L_\infty$ (indeed even $C(\TT)$) whose Fourier series diverge.  In particular, a wonderful result of Kolmogorov \cite{Kolmogorov2} exhibits an $L_1$ function whose Fourier series diverges everywhere.

While convergence in norm is relatively well behaved, pointwise convergence of Fourier series is quite delicate.  Indeed, Kahane and Katznelson \cite{KahaneKatz} showed that for any Lebesgue measure zero set $E\subset\TT$, there exists a continuous function whose Fourier series diverges on $E$.  Thus even continuous functions do not have well-behaved Fourier series.  Nonetheless, Carleson's celebrated theorem \cite{Carleson} shows that the Fourier series for any $f\in L_2$ converges almost everywhere, the analogous result for $L_p$, $1<p<\infty$, being due to Hunt \cite{Hunt}.  For a prolonged discussion on the history of classical results concerning the convergence of Fourier series, the reader is invited to consult the historical notes of \cite[Chapter 4]{Grafakos}.

Returning to the case of $C(\TT)$, while we cannot hope for uniform convergence of the Fourier series of a continuous function, it would be interesting to know if we can permute the terms in the Fourier series prior to taking the partial sums in such a way as to guarantee uniform convergence.  This is the subject of the following problem due to P. L. Ulyanov (for formal definitions see Section \ref{SECDef}).

\begin{problem}[Ulyanov, \cite{Ulyanov}]
Is it true that for every $f\in C(\TT)$, there exists a permutation $\sigma:\Z\to\Z$ such that the rearranged partial sums $S_{\sigma,N}[f]$ converge to $f$ uniformly as $N\to\infty$?
\end{problem}

Despite the apparent difficulty of this problem there are, interestingly enough, two partial results that are indeed quite strong.  The first is due to S. G. \Rev:

\begin{theorem}[\Rev, \cite{revesz}]\label{THMRevesz}
For every $f\in C(\TT)$, there exists a permutation $\sigma:\Z\to\Z$ and a subsequence $\{N_k\}\subset\N$ such that $S_{\sigma,N_k}[f]$ converges to $f$ uniformly as $k\to\infty.$
\end{theorem}

\Rev' argument relies on the fact that the de la Vall\'{e}e Poussin means of a continuous function converge uniformly to the function, and thus it suffices to demonstrate uniform convergence of the difference of $S_{\sigma,N_k}[f]$ and the Poussin mean $V_{N_k}[f]$.  Existence of such a permutation follows via a probabilistic argument.

The second result is due to S. V. Konyagin:

\begin{theorem}[Konyagin, \cite{Konyagin2003,konyagin2008}]\label{THMKonyagin}
If $f\in C(\TT)$ has modulus of continuity satisfying $\omega(f,\delta)=o(1/\log\log(1/\delta)),$ $\delta\to0^+$, then there exists a permutation $\sigma:\Z\to\Z$ such that $S_{\sigma,N}[f]$ converges to $f$ uniformly as $N\to\infty.$
\end{theorem}

Since the Dini--Lipschitz Criterion (e.g. \cite[Chapter II, Theorem 10.3]{Zygmund}) states that $\omega(f,\delta) = o(1/\log(1/\delta))$, $\delta\to0^+$ implies  $S_N[f]$ converges uniformly to $f$, Theorem \ref{THMKonyagin} provides a relatively broad class of functions with modulus of continuity decaying quite slowly which allows uniform convergence of the rearranged Fourier series.

Let us also note that a result of Men\v{s}ov \cite{Mensov} shows that any $f\in C(\TT)$ can be decomposed into the sum of two continuous functions $f_1+f_2$ for which there are two subsequences $\{N_k^{(i)}\}$, $i=1, 2$ such that $S_{N_k^{(i)}}[f_i]$ converges uniformly to $f_i$ for each $i$ (note that this is without rearrangement); however, the subsequences are different in general.

 Additionally, Pogosyan in \cite{Pogosyan1984} constructs a function $f\in L_1(\TT)$ such that any rearrangements of its Fourier series fails to converge in the $L_1$ norm.

 Another facet of the {problem} has been considered by McNeal and Zeytuncu \cite{mcneal2006note}, wherein they demonstrate the existence of a rearrangement $\sigma$ which preserves pointwise convergence of partial sums of Fourier series (i.e. if $S_N[f](t)$ converges, then so does $S_{\sigma,N}[f](t)$). Additionally, they show that convergence of rearranged partial sums under such a permutation need not imply that the non-rearranged partial sums converge.

The purposes of this article are to recast some of the equivalent versions of Ulyanov's problem in the setting of operators on $\mathcal{B}(L_2)$, to discuss some strategies in this setting for proving or disproving the problem, and to provide some results where something can be said about converge on subspaces of $\mathcal{B}(L_2)$.  This method of attack for the problem admits much more flexibility in that there are several natural topologies on $\mathcal{B}(L_2)$ which may be utilized, and we demonstrate that the {problem} may be reduced to considering convergence of multiplication operators induced by the partial sums of the rearranged Fourier series in the strong (or equivalently in this case, weak) operator topology. For the definition of the strong and weak operator topologies see Section \ref{SECDef}.

We briefly state our main results here. If $f\in L_{\infty}(\mathbb T),$ we define a bounded operator $M_{f}\in \mathcal{B}(L_{2}(\mathbb T))$ by $(M_{f}g)(t):=f(t)g(t)$ for all $g\in L_{2}(\mathbb T),t\in \mathbb T.$ Given a bijection $\sigma\colon \Z\to \Z,$ and a natural number $N,$ we define $S_{\sigma,N}[f]\in C(\mathbb T)$ by $(S_{\sigma,N}[f])(t):=\sum_{n=-N}^{N}\widehat{f}(n)e^{2\pi i \sigma(n) t}.$

\begin{theorem}\label{T:main theorem intro}
The following statements are equivalent.
\begin{enumerate}[(i)]
\item Ulyanov's problem has a positive answer.

\item For every $f\in C(\mathbb T),$ there is a bijection $\sigma\colon \Z\to \Z$ so that $M_{S_{\sigma,N}[f]}\to M_{f}$ in the strong operator topology. \label{I:SOT intro}

\item For every $f\in C(\mathbb T)$, there is a bijection $\sigma\colon \Z\to \Z$ so that $M_{S_{\sigma,N}[f]}\to M_{f}$ in the weak operator topology. \label{I:WOT intro}
\end{enumerate}
\end{theorem}

The preceding theorem is interesting not just because it provides an operator theoretic equivalence to Ulyanov's problem, but also because it naturally leads to other operator theoretic questions which are tractable and have interesting solutions in terms of classical and well-studied function spaces. Namely, we have the following two results.

\begin{theorem}\label{T:second theorem intro}
Let $V$ be the set of functions $g\in L_{2}(\mathbb T)$ so that $M_{S_{\sigma,N}[f]}g\to fg$ for all $f\in  L_\infty(\mathbb T),$ and every bijection $\sigma\colon \Z\to\Z.$ Then $V=L_{\infty}(\mathbb T).$ The same conclusion holds if $L_\infty(\TT)$ is replaced with $C(\TT)$ in the definition of $V$.
\end{theorem}

The above theorem characterizes the largest subspace of $L_{2}(\mathbb T)$ on which any rearranged Fourier series converges in the strong operator topology restricted to that subspace.  Note that while Ulyanov's problem is somewhat naturally stated for the space of continuous functions, when speaking of multiplication operators, it makes somewhat more sense to consider the functions defining them to be in the larger space $L_\infty$.  Most of the results moving forward turn out to have the same answer whichever space is used; we indicate when this is the case in the statements of the theorems.

It is not as clear what the weak operator topology version of Theorem \ref{T:second theorem intro} should be. Given a subspace $V$ of $L_2,$ one may consider the weak operator topology coming from $V$ for operators on $L_2.$ This simply means that $T_{n}\to T$ in the weak operator topology coming from $V$ if and only if $\la{T_{n}v,w\ra}\to\la{Tv,w\ra}$ for all $v,w\in V.$ It is clear by Zorn's Lemma that there is \emph{a} maximal subspace $V$ of $L_2$ so that  $M_{S_{\sigma,N}[f]}\to M_{f}$ in the weak operator topology coming from $V$ for every bijection $\sigma\colon \Z\to \Z.$
However, there is no clear reason as to why such a subspace is unique. For example, it is not clear that if $V,W$ have the property that $M_{S_{\sigma,N}[f]}\to M_{f}$ in the weak operator topology coming from $V$ for every bijection $\sigma\colon \Z\to \Z,$ then $V+W$ has this property as well.
Nevertheless, we will in fact show that there is such a maximal subspace and that this subspace is $L_4(\TT).$


\begin{theorem}\label{T:fourth theorem intro}
\begin{enumerate}[(i):]
\item For every $g,h\in L_4(\TT),$ every $f\in L_\infty(\TT),$ and every bijection $\sigma\colon\Z\to \Z$ we have that
\[\la{M_{S_{\sigma,N}}[f]g,h\ra}\to \la{M_{f}g,h\ra}.\]
\item Conversely, let $E$ be a linear subspace of $L_2(\TT),$ and suppose that for every $g,h\in E,$ every $f\in L_\infty(\TT),$ and every bijection $\sigma\colon \Z\to \Z$ we have that
\[\la{M_{S_{\sigma,N}}[f]g,h\ra}\to \la{M_{f}g,h\ra},\]
then $E\subseteq L_4(\TT).$
\item  Statements $(i)$ and $(ii)$ both hold whenever $L_\infty(\TT)$ is replaced with $C(\TT)$.
\end{enumerate}
\end{theorem}

There is a natural dual problem to the above: consider the largest subspace $W$ of $L_\infty(\mathbb T)$ so that $M_{S_{\sigma,N}[f]}\to M_{f}$ in the strong, respectively weak, operator topology for every bijection $\sigma\colon \Z\to \Z.$ In both topologies, the subspace $W$ also ends up being a natural and well-studied object: the Wiener algebra of functions with absolutely summable Fourier coefficients.

\begin{theorem}\label{T:third theorem intro}
\begin{enumerate}[(i)]
\item Let $W$ be the set of $f\in L_\infty(\mathbb T)$ so that $M_{S_{\sigma,N}}[f]\to M_{f}$ in the strong operator topology for every bijection $\sigma\colon \Z\to \Z.$ Then $W=A(\mathbb T).$
\item Let $\widetilde{W}$ be the set of all $f\in L_\infty(\mathbb T)$ so that $M_{S_{\sigma,N}}[f]\to M_{f}$ in the weak operator topology for every bijection $\sigma\colon \Z\to \Z.$ Then $\widetilde{W}=A(\mathbb T).$
\item Likewise, if $W$ and $\widetilde{W}$ are the set of continuous functions whose Fourier series converge unconditionally in the strong, respectively weak, operator topologies, then $W=\widetilde{W}=A(\TT)$.
\end{enumerate}
\end{theorem}

Theorem \ref{T:third theorem intro} is interesting in light of the Dvoretsky-Rogers theorem which states that a Banach space on which every unconditionally convergent series is absolutely convergent is necessarily finite-dimensional. The above theorem shows that even though we cannot hope to deduce absolute convergence from unconditional convergence for a general series in $L_{\infty}(\mathbb T),$ it is nevertheless true that absolute convergence and unconditional convergence are equivalent for Fourier series. One interesting aspect of Theorem \ref{T:third theorem intro} is that $W$ as defined in Theorem \ref{T:third theorem intro} is, in fact, a subalgebra of $C(\mathbb T)$, and this is far from obvious given its definition as the set of functions on which the Fourier series converges unconditionally.


In light of Theorems \ref{T:second theorem intro} and \ref{T:third theorem intro}, it is natural to consider triples $(\Sigma,\W,\V)$ where

\begin{itemize}
\item $\V$ is a linear subspace of $L_{2}(\mathbb T),$
\item $\W$ is a linear subspace of $L_\infty(\mathbb T)$ (or $C(\TT)$),
\item $\Sigma$ is a subset of the symmetric group on $\Z$
\item For every $f\in \W,g\in \V,\sigma\in \Sigma,$ we have $M_{S_{\sigma,N}[f]}g\to fg$.
\end{itemize}
We refer the reader to Section \ref{SECFramework}, for interesting questions in this direction and the difficulty involved in solving them.

The rest of the paper develops as follows: Section \ref{SECDef} contains some basic definitions and preliminaries necessary for the sequel.  In Section \ref{SECEquivalences}, we state our first main theorem, Theorem \ref{THMEquivalences}, which gives several new equivalences to Ulyanov's problem, a sample of which is Theorem \ref{T:main theorem intro}.  In Section \ref{SECRelaxations} we provide some relaxations of the problem based upon some of the criteria in Theorem \ref{THMEquivalences}, which give a roadmap for finding a counterexample to Ulyanov's problem; additionally, we pose some interesting related questions.  Section \ref{SECFramework} exhibits an overarching framework for further investigation into aspects of the rearrangement problem and provides another equivalence to Ulyanov's problem (Remark \ref{REMW}).  Several natural questions arise from our general framework which are further investigated in Sections \ref{SECV}, \ref{SECW}, and \ref{SECWOT}.  Theorems \ref{T:second theorem intro} and \ref{T:third theorem intro} arise from this framework.


For clarity, we present most of the results in dimension 1; however, once the conditions are translated appropriately, everything carries over to arbitrary dimension.  To wit, Section \ref{SECMulti} enumerates the multivariate extensions of the results from previous sections (in particular note that Theorems \ref{T:main theorem intro}, \ref{T:second theorem intro}, and \ref{T:third theorem intro} hold in arbitrary dimension), and the paper ends with a brief conclusion.

\section{Definitions and Basic Facts}\label{SECDef}

Given a function $f\in L_1(\TT)$, define its $n$--th  Fourier coefficient, for $n\in\Z$, via
\[\widehat{f}(n):=\dint_\TT f(t)e^{-2\pi int}dt,\]
and thus the $N$--th partial sum of its Fourier series, for $N\in\N$, via
\[S_N[f](t):=\sum_{|n|\leq N}\widehat{f}(n)e^{2\pi int},\quad t\in\TT.\]
Sometimes for convenience, we will write $e_n$ for the exponential function $e^{2\pi int}$.

Given a permutation of the integers, i.e. a bijection $\sigma:\Z\to\Z$, define the rearranged partial sum of the Fourier series of $f\in L_1(\TT)$ by
\[ S_{\sigma,N}[f](t):=\sum_{|n|\leq N}\widehat{f}(\sigma(n))e^{2\pi i\sigma(n)t},\quad t\in\TT.\]

Our interest in the uniform convergence of the rearranged partial sums will lead us to considering multiplication operators: given $g\in L_\infty(\TT)$, define $M_g:L_2(\TT)\to L_2(\TT)$ via $f\mapsto gf$, and note that $\|M_g\|=\|g\|_{\infty}$.

Of interest to our analysis will be three natural topologies on $\mathcal{B}(\mathcal{H})$ for $\mathcal{H}$ an infinite dimensional, separable Hilbert space: the operator norm topology, the {\em Strong Operator Topology} (SOT), which is the topology of pointwise convergence of operators with $T_n\to T$ in SOT if and only if $\|T_n(h)-T(h)\|_2\to0$ for every $h\in\mathcal{H}$, and finally the {\em Weak Operator Topology} (WOT), in which $T_n\to T$ in WOT if and only if $\bracket{T_n(f)-T(f),g}\to0$ for every $f,g\in\mathcal{H}$. We will also consider WOT and SOT topologies restricted to subspaces of $ \HH $.

To avoid cumbersome terminology, we will at times use the abuse of language to say that a sequence of continuous functions converges in SOT (WOT) to mean that the associated multiplication operators as defined above converge in SOT (WOT).

As a basic tool for finding equivalences to Ulyanov's problem, we use the following well-known fact about convergence in SOT or WOT.

	\begin{theorem} \label{minthm}
		Let $T_N:\HH\rightarrow \HH$ be a sequence of bounded linear operators, $T:\HH\rightarrow \HH$ be another bounded operator, and $V$ be a subset of $\HH$ with dense linear span. Then the following are equivalent:
		\begin{enumerate}[(i)]
			\item $ T_N $ converges to $ T $ in the strong (weak) operator topology of $B(\HH)$.
			\item The norms $\|T_N\|$ are uniformly bounded and the $T_N  $   converges to $T$ in the strong (weak) operator topology on $V $.
		\end{enumerate}
	\end{theorem}
The proof of this theorem is a basic exercise; for example it can be found in \cite[Proposition 1.9.20 and Exercise 1.9.25]{TTAO}.

For notational convenience, we will denote by $S_\infty$ the symmetric group over the integers, i.e. the group of all bijections from $\Z\to\Z$, and elements of this group will be called {\em permutations}, or {\em rearrangements},  of the integers.  Additionally, we will typically write $L_p$ rather than $L_p(\TT)$.  As is customary, $C$ will denote a constant, whose specific value may change from line to line, and subscripts will be used to make dependence of the constant on a given parameter explicit.

\section{Equivalences of the Ulyanov problem}\label{SECEquivalences}

Inspired by the work of \Rev\, we find several equivalent conditions to Ulyanov's problem.

\begin{theorem}\label{THMEquivalences}
The following are equivalent:
\begin{enumerate}[(i)]
\item For every $f\in C(\TT)$, there exists a permutation $\sigma:\Z\to\Z$ such that $\|S_{\sigma,N}[f]-f\|_\infty\to0$ (i.e. Ulyanov's problem has a positive answer).
\item For every $f\in C(\TT)$, there exists a permutation $\sigma:\Z\to\Z$ such that $M_{S_{\sigma,N}[f]}\to M_f$ in the SOT on $\mathcal{B}(L_2)$.
\item For every $f\in C(\TT)$, there exists a permutation $\sigma:\Z\to\Z$ such that $M_{S_{\sigma,N}[f]}\to M_f$ in the WOT on $\mathcal{B}(L_2)$.
\item For every $f\in C(\TT)$, \[\underset{\sigma\in S_\infty}\inf\;\underset{N}\sup\;\|S_{\sigma,N}[f]\|_\infty<\infty.\]
\item There exists a constant $C>0$ such that for every trigonometric polynomial $T$,
\[\underset{\sigma\in S_\infty}\inf\;\underset{N}\sup\;\|S_{\sigma,N}[T]\|_\infty\leq C\|T\|_\infty.\]
\item There exist maps $C:L_2\to\R_+$, and $\sigma:\{\text{trigonometric polynomials}\}\to S_\infty$ such that for every $g\in L_2$ and for every trigonometric polynomial $T$, \label{I:rearrange just trig polys}
\[
\underset{N}\sup\;\|S_{\sigma(T),N}[T]g\|_{2}\leq C(g)\|T\|_\infty.
\]
\item There exist maps $C:L_1\to\R_+$ and $\sigma:\{\text{trigonometric polynomials}\}\to S_\infty$ such that for every $g\in L_1$ and every trigonometric polynomial $T$,
\[\underset{N}\sup\;|\bracket{S_{\sigma(T),N}[T],g}|\leq C(g)\|T\|_\infty. \]
\end{enumerate}
\end{theorem}
\begin{proof}
The proof of the equivalence between conditions $(i),(iv)$ and $(v)$ is due to \Rev~ and can be found in \cite[Section 4]{revesz}.

To prove the equivalence of conditions  $(ii),(iii)$ and $(iv)$, recall that each continuous function defines a multiplication operator $M_f$, and notice that for every $f\in C(\TT)$, every $k,m\in\Z$, and every permutation $\sigma$, we have
\[\|M_{S_{\sigma,N}[f]}e_k-M_fe_k\|_2\to0,\quad \text{and} \quad \bracket{M_{S_{\sigma,N}[f]}e_k-M_fe_k,e_m}\to0,\] as $N\to\infty$.
In other words, for any rearrangement, and any $f\in L_\infty$, we have convergence of the multiplication operators associated with the partial sums in both the SOT and WOT on a set of dense linear span in $L_2$. Therefore, from Theorem \ref{minthm}, for a given  function $f\in L_\infty$, the rearranged partial sums $S_{\sigma,N}[f]$ converge to $f$ in the strong or  (in this case) equivalently weak operator topologies if and only if $\|S_{\sigma,N}[f]\|_{\infty}$ are uniformly bounded in $N$.

Obviously, condition $(v)$ implies both $(vi)$ and $(vii).$

To prove condition $(vi)$ implies $(ii)$, let $f\in C(\TT)$; from Theorem \ref{THMRevesz}, there exists a permutation $\sigma:\Z\to\Z$ and a subsequence $\{N_k\}\subset\N$ such that $S_{\sigma,N_k}[f]$ converges to $f$ uniformly as $k\to\infty.$ Let $g\in L_2(\TT)$, and put $T_{k}=S_{\sigma,N_{k+1}}[f]-S_{\sigma,{N_k}}[f]$.  Then by assumption, we have that for some rearrangement $\sigma_k$,
\[\underset{N}\sup\;\|S_{\sigma_k,N}[T_{k}]g\|_{2}\leq C(g)\|S_{\sigma,N_{k+1}}[f]-S_{\sigma,{N_k}}[f]\|_\infty.\]
Note that we may assume that $\sigma_{k}(\sigma(\{N_{k}+1,\cdots,N_{k+1}\}))=\sigma(\{N_{k}+1,\cdots,N_{k+1}\}).$ Indeed, the $j$-th Fourier coefficient of $T_{k}$ for $j\notin \sigma(\{N_{k}+1,\cdots,N_{k+1}\})$ is zero. So condition $(\ref{I:rearrange just trig polys})$ is still true if we replace $\sigma_k$ by a permutation $\widetilde{\sigma_{k}}$ with the property that $\widetilde{\sigma_{k}}(\sigma(N_{j}))=\sigma_{k}(l_{j})$ where $(l_{j})_{j=N_{k}+1}^{N_{k+1}}$ are the unique increasing sequence of integers satisfying
\[|\{\sigma_{k}(l_{N_{k}+1}),\cdots,\sigma_{k}(l_{N_{k}+j})\}\cap \sigma_{k}(\{N_{k}+1,\cdots,N_{k+1}\})|=j.\]
Hence we may, and will, assume that
\[\sigma_{k}(\sigma(\{N_{k}+1,\cdots,N_{k+1}\}))=\sigma_{k}(\sigma(\{N_{k}+1,\cdots,N_{k+1}\})).\]
Denote by $\bar \sigma$ the rearrangement of $\Z$  defined as follows: if $j\in (N_{k},N_{k+1}]$, then $\bar\sigma(j) = \sigma_{k}\circ\sigma(j)$, and $\bar\sigma(j)$ is the identity outside of these intervals.  Evidently, $\bar\sigma$ is a rearrangement of $\Z$.
Moreover, for every $j,j^\prime\in\N$, if $  N_k<j\leq N_{k+1} $ and $N_{k^\prime}<j^\prime\leq N_{k^\prime+1}$, then
$\|S_{\bar\sigma,j}[f]g-S_{\bar\sigma,j^\prime}[f]g\|_2$ is majorized by
\[ \|S_{\bar\sigma,N_{k}}[f]g-S_{\bar\sigma,j}[f]g\|_2+\|S_{\bar\sigma,N_{k^\prime}}[f]g-S_{\bar\sigma,j^\prime}[f]g\|_2+\|S_{\bar\sigma,N_{k}}[f]g-S_{\bar\sigma,N_{k^\prime}}[f]g\|_2,\]
which is at most
\begin{multline} C(g)\|S_{\sigma,N_{k+1}}[f]-S_{\sigma,{N_k}}[f]\|_\infty+C(g)\|S_{\sigma,N_{k^\prime+1}}[f]-S_{\sigma,{N_k^\prime}}[f]\|_\infty\nonumber\\ +\|S_{\sigma,N_{k}}[f]-S_{\sigma,N_{k^\prime}}[f]\|_\infty \|g\|_2.  \end{multline}
Thus, $\{S_{\bar\sigma,j}[f]g\}$ is Cauchy in $L_2$, and since $\{S_{\bar\sigma,N_k}[f]g\}$ converges to $fg$, it follows that  $ \|S_{\bar\sigma,j}[f]g-fg\|_2 \to 0$, which gives $(ii)$ as $g$ was arbitrary.
The fact that $(vii)$ implies $(iii)$ can be proved similarly after we make the observation  that, for $f,g\in L_2$, $fg\in L_1$.
\end{proof}

One advantage of the conditions given above is that conditions \textit{(ii)} and \textit{(iii)} are much weaker than condition \textit{(i)} for general operators.  That is, the strong and weak operator topologies are weaker than the operator norm topology, which corresponds to uniform convergence of the partial sums.  The weak operator topology is the weakest of the three, so it is hoped that convergence of the rearranged series will be easier to prove in this topology.

With the aid of this theorem, we will develop some new strategies that may lead to a proof of or counterexample to Ulyanov's problem.  This approach also allows relaxations of the problem to more tangible problems which will be discussed in the sequel.

It is unclear at the moment if Theorem \ref{THMEquivalences} remains true if $C(\TT)$ is replaced with $L_\infty$ in conditions $(ii)$ and $(iii)$.

\section{Relaxations of the Ulyanov's problem}\label{SECRelaxations}

We refer the reader to the work of \Rev~ and Konyagin for a discussion of strategies for proving Ulyanov's problem which arise from classical Fourier series methods.  Here, we consider some alternative approaches based on the multiplication operator equivalence to the problem discussed above.

Recall that Ulyanov's problem is equivalent to the existence of a permutation such that $M_{S_{\sigma,N}[f]}\to M_f$ in the SOT (equivalently the WOT) on $\mathcal{B}(L_2)$.  To make the problem more tangible, these observations suggest that we consider smaller subspaces of $L_2$ such that the rearranged multiplication operators converge in the SOT restricted to that subspace.

To discuss convergence on a subspace of $L_2$, we first consider the following definition, which is essentially that condition $(vi)$ in Theorem \ref{THMEquivalences} holds for a subset $S\subset L_2$.

\begin{definition}\label{DEFrearrangementprop}
A subset $\mathcal{S}\subset L_2$ is said to have the {\em bounded rearrangement property} if there exist maps $C:\mathcal{S}\to\R_+$, and $\sigma:\{\text{trigonometric polynomials}\}\to S_\infty$ such that for every $g\in \mathcal{S}$ and for every trigonometric polynomial $T$,
\begin{equation}\label{weakestim}
\underset{N}\sup\;\|S_{\sigma(T),N}[T]g\|_{2}\leq C(g)\|T\|_\infty.
\end{equation}
\end{definition}

Note that if $\mathcal{S}$ has the bounded rearrangement property, then so does its linear span; however, something stronger holds.

\begin{proposition}\label{extensionsofset}
Suppose $\mathcal{S}$ has the bounded rearrangement property.
\begin{enumerate}[(i)]
 \item Then the linear span of $\mathcal{S}$ has the bounded rearrangement property.
\item Let $\mathcal{U}$ be a subset of  $L_\infty$ such that for every $g\in \mathcal{U}$ there exist $h,h_1\in \mathcal{S}$ with $|g(t)-h(t)|<h_1(t)$ for almost every $t\in\TT$.  Then $\mathcal{U}$ has the bounded rearrangement property.
\item  The following set also has the bounded rearrangement property: \[\left\{g\left(\frac{t+j}{k}\right)\;:\;g\in S, k\in\N,0\leq j\leq k\right\}.\]
\end{enumerate}
\end{proposition}

\begin{proof} $(i)$:
If $g=\sum_{j}a_jh_j$, $ h_j\in\mathcal{S} $, then simply take $C(g)=\sum_j|a_j|C(h_j)$.
	
	$(ii)$: Suppose $g\in \mathcal{U}$, $ h, h_1\in \mathcal{S}$ and $|g(t)-h(t)|<h_1(t)$ for almost every $t\in\TT$. Then for any trigonometric polynomial and the corresponding permutation  $\sigma$ of $ \Z $,
	\begin{align*}
	\| S_{\sigma,N}[T]g\|_2 &\leq \| S_{\sigma,N}[T](g-h)\|_2+\| S_{\sigma,N}[T]h\|_2\\ &\leq
	 \| S_{\sigma,N}[T]h_1\|_2+\| S_{\sigma,N}[T]h\|_2\leq (C(h_1)+C(h))\|T\|_\infty\\ &=:C(g)\|T\|_\infty.\end{align*}
	
	$(iii)$: For a given trigonometric polynomial $T$, let $\tilde{T}=T(kt)$. Then
\begin{align*}
C(g)^2\|T\|_\infty^2=C(g)^2\|\tilde T\|_\infty^2\geq \|S_{\sigma(\tilde{T}),N}[\tilde{T}]g\|^2_2& =\int\limits_{\TT}|S_{\sigma(\tilde{T}),N}[\tilde{T}](t)|^2|g(t)|^2dt\\
& = \frac{1}{k}\int\limits_{k\TT}\left|S_{\sigma(\tilde{T}),N}[\tilde{T}]\left(\frac{u}{k}\right)\right|^2\left|g\left(\frac{u}{k}\right)\right|^2du,\\
\end{align*}
which is at least
\[
\frac{1}{k}\int\limits_{\TT}\left|S_{\sigma(\tilde{T}),N}[\tilde{T}]\left(\frac{u+j}{k}\right)\right|^2\left|g\left(\frac{u+j}{k}\right)\right|^2du=\frac{1}{k}\left\|S_{\sigma(\tilde{T}),N}[T]g\left(\frac{\cdot+j}{k}\right)\right\|^2_2,\]
where we used the fact that $ \tilde T $ is $ \frac{1}{k} $ periodic.  Consequently, we may take $\sigma(T)=\sigma(\tilde{T})$ and $C\left(g\left(\frac{\cdot+j}{k}\right)\right)=\sqrt{k}C(g)$.
\end{proof}

\begin{corollary} Suppose a subset $\mathcal{S}$ of $L_2$ has the bounded rearrangement property, then the following sets also have this property:
\begin{enumerate}[(i)]
\item The set
$$\{h\in L_2(\TT):\;|h(t)|\leq |g(t)|\text{ for all } t\in \TT \text{ and some }g\in \mathcal{S}\},$$
\item The closed linear span of $\mathcal{S}$ in the $L_\infty$ norm,
\item The set $$\{gh:\;g \text{ in closed linear span of }\mathcal{S}\text{ in } L_\infty \text{ norm, }h\in L_\infty(\TT)\}.$$
\end{enumerate}
\end{corollary}

Note that part (iii) of the above Corollary follows from parts (i) and (ii).

Related to Proposition \ref{extensionsofset}, we suggest the following problem.

\begin{problem} \label{prob:brg}Does every countable $\mathcal{S}\subset L_2$ have the bounded rearrangement property?
\end{problem}

Unfortunately, having a positive answer to the above problem will not prove Ulyanov's problem since $L_2$ cannot be constructed from a countable set using its extensions from the Proposition \ref{extensionsofset}. However, if Ulyanov's problem has a positive answer, Problem \ref{prob:brg} may be regarded as its relaxation, which may be easier to prove. It may also happen that Problem \ref{prob:brg} has a positive answer while Ulyanov's problem does not. 

Notice that the bounded rearrangement property on a subspace $\mathcal{S}$ implies that for every $f\in C(\TT)$, there exists a permutation $\sigma:\Z\to\Z$ such that $M_{S_{\sigma,N}[f]}\to M_f$ in the SOT on the subspace $\mathcal{S}$.
\begin{problem}
Let $\mathcal{S}$ be a subspace of $L_2(\TT). $ Suppose every function  $f\in C(\TT)$ has a rearrangement of its  Fourier series that converges to $f$ in the SOT on $\mathcal{S}$. Does $\mathcal{S}$ satisfy the bounded rearrangement property?
\end{problem}

A similar result to Proposition \ref{extensionsofset} can be proved for the weak operator topology for the closure in the $L_2(\TT)$ norm.

\begin{proposition}\label{extensionsofsetweak}
Suppose $\mathcal{S}\subset L_1$, and there exist maps $C:\mathcal{S}\to\R_+$ and $\sigma:\{\textnormal{trigonometric polynomials}\}\to S_\infty$ such that 
\begin{equation}\label{weakestimtrigpol}
\underset{N}\sup\;|\bracket{S_{\sigma(T),N}[T],g}|\leq C(g)\|T\|_\infty.
\end{equation}
\begin{enumerate}[(i)]
 \item Then \eqref{weakestimtrigpol} also holds for the linear span of $  \mathcal{S} $ .
\item Let $\mathcal{U}$ be a subset of $L_1$ such that for every $g\in \mathcal{U}$ there exists $h\in \mathcal{S}$ with $\|g(t)-h(t)\|_2<\infty$. Then  \eqref{weakestimtrigpol} also holds for $\mathcal{U}$.
\end{enumerate}
\end{proposition}

The proof of this proposition follows by similar reasoning to that of Proposition \ref{extensionsofset} and so is omitted.

\section{Framework for further investigation}\label{SECFramework}
In this section we investigate the relationships between the set of permutations, the class of functions whose rearranged partial Fourier series we are considering and the subspace  on which we want to have convergence in the strong operator topology.  In particular we notice algebraic structures indirectly appearing in our analysis.

We consider triples $(\Sigma,\mathcal{F},\mathcal{V})$ where $\Sigma$ is a set of permutations, $\mathcal{F}\subset L_\infty(\TT)$  and  $\mathcal{V}$ is a {subset} of $ L_2(\TT)$, for which
\begin{equation}\label{funcpermsub}\|S_{\sigma,N}[f]g-fg\|_2\to0,\quad  \text{for every } g\in \mathcal{V}, \text{ every }\sigma\in\Sigma,  \text{ and every } f\in\mathcal{F}.\end{equation}

Given a set of permutations $\Sigma\subset S_\infty$, let $\mathcal{V}_\Sigma$ be the largest subset of $L_2(\TT)$ such that \eqref{funcpermsub} holds for the triple $(\Sigma,L_\infty(\TT),\mathcal{V}_\Sigma)$  and let $\mathcal{W}_\Sigma$ be the largest subset of $C(\TT)$ such that \eqref{funcpermsub} holds for the triple $(\Sigma,\mathcal{W}_\Sigma, L_2(\TT))$.

A couple of notes are in order.  First, it is easily seen that both $\mathcal{V}_\Sigma$ and $\mathcal{W}_\Sigma$ are subspaces of $L_2$ and $C(\TT)$, respectively, regardless of the class of permutations $\Sigma$. For  $\mathcal{V}_\Sigma$  it is obvious. The fact that $\W_{\Sigma}$ is a subspace of $C(\TT)$ follows directly from linearity of $S_{\sigma,N}$, the triangle inequality, and the inequality $(a+b)^2\leq 2a^2+2b^2$. Second, notice that if $\Sigma_1\subset\Sigma_2$, then $\V_{\Sigma_1}\subset\V_{\Sigma_2}$.  Therefore, one item of interest is to consider what the largest such space, $\V_{S_\infty}$, is.  This is the subject of the investigation in Section \ref{SECV}, where we determine that $\V_{S_\infty}=L_\infty$.

\begin{remark}\label{REMW}
Regard that Ulyanov's problem is equivalent to showing that \[\underset{\Sigma\subset S_\infty}\bigcup \W_\Sigma = C(\TT).\]  Of course in all of these matters, one should keep in mind that $\Sigma$ should be viewed as a set of equivalence classes in $S_\infty$, where two permutations are equivalent if they differ by at most finitely many coordinates.  Also notice that if $\Sigma_1\subset\Sigma_2$, then $\W_{\Sigma_2}\subset\W_{\Sigma_1}$.  Hence, it would suffice to show that
\[\underset{\sigma\in S_\infty}\bigcup \W_{\{\sigma\}} = C(\TT).\]
\end{remark}

Note that here we define $\mathcal{W}_\Sigma$ to be a subset of $C(\TT)$ because this is the setting in which it may shed light onto Ulyanov's problem.  However, in the sequel, we demonstrate that the characterization of $\mathcal{W}_{S_\infty}$ remains the same when $C(\TT)$ is replaced by $L_\infty$.

Unfortunately, even determining $\W_{\{\sigma\}}$ for a fixed $\sigma\in\Sigma$ is elusive.  Indeed, for $\sigma$ being the identity we do not know of a characterization. 
In conclusion, we pose the following general problem for further investigation.

\begin{problem}
Characterize the spaces $\mathcal{V}_\Sigma$ and $\mathcal{W}_\Sigma$ for general sets of permutations.
\end{problem}

One final note: as a similar line of investigation to finding $\V_{S_\infty}$, we will demonstrate that $\W_{S_\infty}=A(\TT)$, the Wiener algebra of functions with absolutely convergent Fourier series.  In this case, $\W_\Sigma$ is a subalgebra of $C(\TT)$; however, we do not have a direct proof of this fact because we provide the characterization by other means and obtain as a corollary that $\W_{S_\infty}$ is an algebra.  Consequently, this naturally leads to the following question.

\begin{problem}
Under what conditions on the set $\Sigma$ is $\W_\Sigma$  a subalgebra of $C(\TT)$?
\end{problem}

As a particular case of this problem, take $\Sigma=\{\text{id}\}$. It turns out that $\W_{\{\text{id}\}}$ is not closed under multiplication, and so is not an algebra.  To see this, note that $\W_{\{\text{id}\}}$ is the set of all $f\in C(\TT)$ such that $\|S_N[f]\|_\infty$ is uniformly bounded in $N$.  Then the proof of the main theorem in \cite{olevskialg} implies that the smallest algebra containing the functions with uniformly bounded partial Fourier sums coincides with $C(\TT)$, hence $\W_{\{id\}}$ is not an algebra.

\section{Characterization of $\mathcal{V}_{S_\infty}$}\label{SECV}

In this section we provide a full characterization of the subspace $\V_{S_\infty}$, i.e. the largest subspace on which the Fourier series converge unconditionally in SOT for every continuous function.

As a preliminary note, consider the following.

\begin{proposition}\label{PROPgLinfty}
Let $f\in L_\infty$.  Suppose there exists a permutation $\sigma$ and a $2\leq q\leq\infty$ such that $\|S_{\sigma,N}[f]-f\|_q\to0$.  Then for every $g\in L_p$ and for $\frac1p=\frac12 (1-\frac2q)$, \[\|S_{\sigma,N}[f]g-fg\|_2\to0.\]
\end{proposition}
\begin{proof}
Apply H\"{o}lder's inequality with $\frac q2$ and its conjugate $\frac{1}{s}=1-\frac2q$ to see that
\[\|S_{\sigma,N}[f]g-fg\|_2^2\leq\|S_{\sigma,N}[f]-f\|_q^2\|g\|_{2s}^2.\]
Setting $p=2s$ provides the desired conclusion.
\end{proof}

\begin{proposition}\label{CORgLinfty}
For any $f\in L_\infty$, any permutation $\sigma$ of $\Z$, and for every $g\in L_\infty$,
\[\|S_{\sigma,N}[f]g-fg\|_2\to0.\]
\end{proposition}

\begin{proof}
The norm in question is majorized by $\|S_{\sigma,N}[f]-f\|_2\|g\|_\infty$ which converges for any permutation since the exponential system is an orthonormal basis for $L_2$.
\end{proof}

The consequences of this proposition are two-fold.  First, it shows that if we are to find a counterexample to Ulyanov's problem (specifically to condition \textit{(ii)} in Theorem \ref{THMEquivalences}), we must search for a $g\in L_2\setminus L_\infty$ for which $\|S_{\sigma,N}[f]g-fg\|_2$ does not converge to $0$ for any permutation $\sigma$.  Secondly, it implies that $L_\infty\subset \V_{S_\infty}$, but we will demonstrate that in fact $\V_{S_\infty}=L_\infty$.

\begin{theorem}\label{THMVLinfty}
$\V_{S_\infty}=L_\infty(\TT)$.
\end{theorem}

The proof of this theorem will be accomplished in a series of intermediate steps, and we will pose some interesting problems along the way.

First, given a sequence of signs $\Epsilon:=\{\eps_n\}_{n\in\Z}\in\{-1,1\}^\Z$, and $f\in C(\TT)$ with Fourier coefficients $(c_n)$, define
\[T_{\Epsilon,N}[f](t):=\sum_{|n|\leq N}\eps_nc_ne^{2\pi int},\]
and let $T_{\Epsilon}[f]$ be the $L_2$ limit as $N\to\infty$ of $T_{\Epsilon,N}[f]$ (whose existence is guaranteed by basic considerations).  Next, given a $g\in L_2$, denote by $L_{2,g}$ the Hilbert space of measurable functions on $\supp(g)$  equipped with the inner product $\bracket{f_1,f_2}_{L_{2,g}}:=\bracket{gf_1,gf_2}_{L_2}.$ Note that, for functions in $C(\TT)$, convergence in SOT is equivalent to  convergence  in $L_{2,g}$ for every $g\in L_2$.

\begin{proposition}\label{equivvstar}
For a given $g\in L_2$, the following are equivalent:
\begin{enumerate}[(i)]
\item $S_{\sigma,N}[f]$ converges in $L_{2,g}$ for every permutation $\sigma$ and every $f\in C(\TT)$.
\item For every $\Epsilon\in\{-1,1\}^\Z$ and every $f\in C(\TT)$, $T_{\Epsilon,N}[f]$ converges in $L_{2,g}$.
\item There exists a constant $C_g$ independent of $\Epsilon, N$, and $f\in C(\TT)$, such that
\[\|T_{\Epsilon,N}[f]\|_{L_{2,g}}\leq C_g\|f\|_\infty.\]
\end{enumerate}
\end{proposition}
\begin{proof}
$ (i)\Longleftrightarrow (ii) $ follows from  $ I, \S 1, $  Theorem 1 in \cite{kashsahak}, for example.

$(ii)\implies (iii)$: Suppose there is no such $C_g$. Using the fact that $\|T_{\mathcal{E}, N}[f]\|_{L_{2,g}}\leq (2N+1) \|f\|_\infty\|g\|_2$, there exist integers $N_k<M_k<N_{k+1}$ $k=1,2,\dots$,  coefficients $c_n$, and signs $\eps_n\in \{-1,1\}$, such that
	\[\left\|\sum_{N_k\leq |n|\leq M_k}c_ne^{2\pi i n t}\right\|_\infty=1\;\;\text{and}\,\; \left\|\sum_{N_k\leq |n|\leq M_k}\eps_nc_ne^{2\pi i n t}\right\|_{L_{2,g}}\geq k^2.\]
	Then let
	$$f(t):=\sum_{k=1}^\infty\frac{1}{k^2}f_k(t):=\sum_{k=1}^{\infty}\frac{1}{k^2}\sum_{N_k\leq |n|\leq M_k}c_ne^{2\pi i n t}.$$
	The series above is uniformly convergent, and thus $f$ is in $C(\TT)$. However, $T_{\mathcal{E}, N}[f]$ is not convergent in $L_{2,g}$ as it fails to be Cauchy.  Indeed, note that $T_{\Epsilon,M_k}[f_j]-T_{\Epsilon,N_k}[f_j]$ is 0 if $j\neq k$ and $\sum_{N_k\leq|n|\leq M_k}\eps_nc_ne^{2\pi int}$ if $j=k$.  Consequently, by linearity of $T_{\Epsilon,N}$ and the assumption on the $L_{2,g}$ norm of the given function, $\|T_{\Epsilon,M_k}[f]-T_{\Epsilon,N_k}[f]\|_{L_{2,g}}\geq1$ for all $k$.
	
		$(iii)\implies (ii)$: Let $\delta>0$, $\Epsilon\in\{-1,1\}^\Z$, and $f\in C(\TT)$ be fixed but arbitrary. By Theorem \ref{THMRevesz}, there exists a permutation $\sigma$ of $\Z$ and an integer $N>0$ such that \[\|S_{\sigma,N}[f]-f\|_\infty<\frac{\delta}{2C_g}.\]
Put $P=S_{\sigma,N}[f]$ and let  $$ n>m\geq \max_{ |k|\leq N}{\sigma(k)} .$$
Then $T_{\mathcal{E},n}[P]=T_{\mathcal{E},m}[P]$, hence
\begin{align*}
\|T_{\mathcal{E},n}[f]-T_{\mathcal{E},m}[f]\|_{L_{2,g}} & \leq \|T_{\mathcal{E},n}[f]-T_{\mathcal{E},n}[P]\|_{L_{2,g}}+\|T_{\mathcal{E},m}[f]-T_{\mathcal{E},m}[P]\|_{L_{2,g}} \\
 &= \|T_{\mathcal{E},n}[f-P]\|_{L_{2,g}}+\|T_{\mathcal{E},m}[f-P]\|_{L_{2,g}}\leq 2C_g\|P-f\|_\infty\\
 & <\delta,\\
\end{align*}
which completes the proof as $\delta$ was arbitrary.		
\end{proof}

\begin{remark}\label{REMVstarLinfty}
Note that if the conditions in the previous theorem are satisfied then for any permutation $ \sigma $, $S_{\sigma, N}[f]$ converges to $f$ in $L_{2,g}$, hence $g\in \V_{S_\infty}$.
\end{remark}

\begin{proposition}\label{equivLinfit}The following are equivalent
	\begin{enumerate}[(i)]
\item $\V_{S_\infty}=L_\infty(\TT)$.
\item There exists an absolute constant $c>0$ such that for any measurable set $E\subseteq \TT$ with $ |E|>0 $ there exists a function $f\in C(\TT)$ with $ \|f\|_\infty\leq 1 $, and there exists a sequence $\mathcal{E}=\{\eps_n\}\in\{-1,1\}^\Z$ for which
$$  \|T_{\mathcal{E}}[f]\mathbbm{1}_E\|_{2}\geq c.$$
	\end{enumerate}
\end{proposition}
\begin{proof} $ (i)\implies(ii)$:  If $(ii)$ is false, then for every $k\in\N$ there exists a set $E_k$ with $ |E_k|>0 $ such that for any function $f\in C(\TT)$ with $\|f\|_\infty\leq 1$, any sequence  $\mathcal{E}\in\{-1,1\}^\Z,$ and any $N\in\N,$
\begin{equation}\label{E:smallest condition for V} \|T_{\mathcal{E},N}[f]\mathbbm{1}_{E_k}\|_{2}\leq \frac{1}{k^2}.
\end{equation}
	We claim that we can assume, without loss of generality, that the sets $E_k$ are mutually disjoint, and that $|E_k|\leq \frac{1}{k^4}$. To see this, suppose that $(E_{k})_{k}$ are any sets satisfying (\ref{E:smallest condition for V}). First note that if we replace any of the $E_{k}$'s with a subset of positive measure, then they still satisfy (\ref{E:smallest condition for V}). It is a simple measure-theoretic argument to show that we may find a subset $F_0\subseteq E_0$ with $|F_0|>0,$ and so that $I_{0}=\{k\in \N: |E_{k}\cap (\TT\setminus F_0)|>0\}$ is infinite. Fix a $k_{1}\in I_0,$ by another simple measure-theoretic argument, we may find a subset $F_{k_1}\subseteq E_{k_1}$ with $|F_{k_{1}}|>0$ and so that $I_{1}=\{k\in I_0:k>k_{1}, |E_{k}\cap (\TT\setminus F_{k_1})|>0\}$ is infinite. Continuing in this way, we find a decreasing sequence $I_{l}$ of infinite subsets of $\Z,$ an increasing sequence $k_{l}$ of integers, and a sequence $F_{k_{l}}\subseteq E_{k_{l}}$ so that: $k_{l}\in I_{l-1},$ so that $|F_{k_{l}}|>0,$ and so that for all $r\in I_{k_{l}}$ we have that $|E_{r}\cap (\TT\setminus F_{k_{l}})|>0.$ Replacing the $E_{k}$ with $F_{k_{l}}$ and then re-indexing and renaming, we may assume that the $E_k$ are disjoint.

    Now consider the function
	\[ g=\sum_{k=1}^{\infty}k\cdot \mathbbm{1}_{E_k}. \]
	Note that $g\in L_2\setminus L_\infty$, and that for every $f\in C(\TT)$ with $\|f\|_\infty\leq 1$, any sequence of signs  $\mathcal{E}$ and any $N\in\N$,
\[ \|T_{\mathcal{E},N}[f]\|^2_{L_{2,g}}= \sum_{k=1}^{\infty}  k^2\int_{E_k}|T_{\mathcal{E},N}[f]|^2dt\leq\sum_{k=1}^{\infty}  \frac{1}{k^2}<\infty. \]
 	Consequently, there is a constant $C$ (in particular $C=\pi^2/6$) such that for every $\Epsilon$, $f\in C(\TT)$ with $\|f\|_\infty\leq1$, and every $N\in\N$, $\|T_{\Epsilon,N}[f]\|_{L_{2,g}}\leq C$, whence Proposition \ref{equivvstar} and Remark \ref{REMVstarLinfty} imply that $S_{\sigma,N}[f]$ converges to $f$ in $L_{2,g}$ for any permutation $\sigma$.  It follows that $g\in \V_{S_\infty}\setminus L_\infty$, which contradicts $(i)$.
 	
$ (ii)\implies (i)$: Suppose $g\notin L_\infty$. Then for every positive integer $k$ there exists a set $E_k$ of positive measure such that $|g(t)|>k$ for almost every $t\in E_k$. By assumption, there must be a function $f_k\in C(\TT)$ with $ \|f_k\|_\infty\leq 1 $, and a sequence $\mathcal{E}\in \{-1,1\}^\Z $  for which
\[ \|T_{\mathcal{E}}[f_k]\mathbbm{1}_{E_k}\|_{2}\geq c\]
 	where $c$ is independent of $k$. Then for sufficiently large $N_k$, we have
\[ \|T_{\mathcal{E},N_k}[f_k]\mathbbm{1}_{E_k}\|_{2}\geq \frac{c}{2}.\]	
Thus,
\[ \|T_{\mathcal{E},N_k}[f_k]\|_{L_{2,g}}\geq k^2\|T_{\mathcal{E},N_k}[f_k]\mathbbm{1}_{E_k}\|_{2}\geq \frac{ck^2}{2},\]
which, by the implication $(i)\Rightarrow (iii)$ of Proposition \ref{equivLinfit}, implies that $g\notin \V_{S_\infty}$.
\end{proof}

Note that we will prove Theorem \ref{THMVLinfty} by proving condition \textit{(ii)} of Proposition \ref{equivvstar}.  On our way to this, let us pause to consider some side problems which are of independent interest.

\begin{problem}\label{PROBC0}
What is the smallest constant $C_0\geq1$ such that there exists an $(\eps_n)_{n=0}^\infty\subset\{-1,1\}$ such that $\|\finsum{n}{0}{N}\eps_ne^{2\pi in\cdot}\|_{\infty}\leq C_0\sqrt{N+1}$ for { infinitely many $N$}?\footnote{We thank A. Sahakian for pointing out to us that $C_0\leq5$ due to \cite[Theorem 11, p.130]{kashsahak} and that this problem can be used to prove condition $(ii)$ of Proposition \ref{equivvstar}.}
\end{problem}

\begin{remark}\label{REMC0}
 Problem \ref{PROBC0} is a special case of {\em Littlewood's Problem} in $L_\infty$, for which the Rudin--Shapiro polynomials give an upper bound of $C_0\leq\sqrt{2}$ \cite[Chapter 4, Exercise 1]{Borwein}\footnote{We thank T. Erd\'{e}lye for pointing out this example to us.}. Meanwhile, $C_0\geq1$ because $\|\finsum{n}{0}{N}\eps_ne^{2\pi in\cdot}\|_{\infty}\geq\|\finsum{n}{0}{N}\eps_ne^{2\pi in\cdot}\|_{2} = \|(\eps_n)_{n=0}^N\|_{\ell_2} = \sqrt{N+1}$, where the equality holds by Parseval's Theorem.
\end{remark}

\begin{problem}\label{PROBC1}
Find the largest constant $C_1>0$ such that for all $0<c<C_1$ and for every measurable $E\subset\TT$ with $|E|>0$, there exists an $f\in C(\TT)$ with $\|f\|_{\infty}\leq1$ and an $(\eps_n)_{n\in\Z}\subset\{-1,1\}$ such that $\|\dsum_{n\in\Z}\eps_nc_ne^{2\pi in\cdot}\mathbbm{1}_E\|_{2}\geq c$?
\end{problem}

Of course, if $C_1>0$, then in particular, Proposition \ref{equivLinfit} item (ii) has a positive answer, which implies that $\V_{S_\infty}=L_\infty$.  However, we prove something stronger in the following theorem by giving a lower bound on $C_1$ in terms of $C_0$.

\begin{theorem}\label{THMC0C1}
Letting $C_0$ and $C_1$ be as in Problems \ref{PROBC0} and \ref{PROBC1}, we have $$C_1\geq\dfrac{2\sqrt{2}}{3\sqrt{3\pi}}\dfrac{1}{C_0}.$$
\end{theorem}

\begin{proof}
Let $0<\gamma_0<1$ and $E\subset\TT$ with $|E|>0$ be fixed but arbitrary.  Without loss of generality, we may assume that $0$ is a point of Lebesgue density for $E$; if not, then translate the function $f$ constructed below.  Consequently, $\lim_{\gamma\to0}|E\cap[-\gamma,\gamma]|/(2\gamma)=1.$  Therefore, there exists an $\gamma_1>0$ such that for every $\gamma<\gamma_1$,
\begin{equation}\label{EQLebesgueDensity}1\geq|E\cap[-\gamma,\gamma]|\geq2\gamma(1-\gamma_0)^2.
\end{equation}
Next, note that for every $t\in\R$, $|\sin t/t|\geq1-|t|$. Consequently, if $0<\delta<1$ is arbitrary, then
\begin{equation}\label{EQsinc}
\left|\frac{\sin t}{t}\right|\geq1-\delta,\quad |t|\leq\delta.
\end{equation}
Choose $N\in\N$ large enough so that $\frac{\delta}{\pi(N+1)}<\gamma_1$, and let $(\eps_n)_{n\in\Z}\subset\{-1,1\}$ be a sequence satisfying the conclusion of Problem \ref{PROBC1} with constant $C_0$ (note that by remark \ref{REMC0}, such a sequence exists which gives an upper bound of $\sqrt{2}\sqrt{N+1}$).  Subsequently, let
$$f(t):=\dfrac{1}{C_0\sqrt{N+1}}\finsum{n}{0}{N}\eps_ne^{2\pi int}.$$  Evidently, $\|f\|_{\infty}\leq1$, and it remains to consider the norm of $T_{\Epsilon}[f] = \frac{1}{C_0\sqrt{N+1}}\finsum{n}{0}{N}e^{2\pi in\cdot}$.  It is readily verified that $\finsum{n}{0}{N}e^{2\pi int}=e^{i\pi N}\sin(\pi(N+1)t)/\sin(\pi t)$, whence $$T_{\Epsilon}[f](t) = \dfrac{1}{C_0\sqrt{N+1}}(-1)^N\dfrac{\sin(\pi(N+1)t)}{\sin(\pi t)}.$$

It follows from \eqref{EQsinc} that for all $|t|\leq\delta/(\pi(N+1))$, $$|T_{\Epsilon}[f](t)|\geq\dfrac{1}{C_0\sqrt{N+1}}(1-\delta)(N+1)\geq\dfrac{\sqrt{N+1}}{C_0}(1-\delta).$$  Therefore, the norm of $T_{\Epsilon}[f]$ restricted to $E$ is bounded below by its norm restricted to $E\cap[-\delta/(\pi(N+1)),\delta/(\pi(N+1))]$, which is at least
$$\dfrac{\sqrt{N+1}}{C_0}(1-\delta)|E\cap[-\delta/(\pi(N+1)),\delta/(\pi(N+1))]|^\frac12.$$  Therefore, on account of \eqref{EQLebesgueDensity} and the choice of $N$,
$$\|T_{\Epsilon}[f]\mathbbm{1}_E\|_{L^2(\TT)}\geq\dfrac{\sqrt{N+1}}{C_0}(1-\delta)\sqrt\frac{2\delta}{\pi(N+1)}(1-\gamma_0).$$  Since $\delta$ is arbitrary, we can maximize the right hand side above by choosing $\delta=1/3$, which is where the maximum of the function $x\mapsto(1-x)\sqrt{x}$ is attained.  Thus the bound on $C_1$ follows since $\gamma_0$ is arbitrarily close to $0$.  In particular, for any $c<\frac{2\sqrt{2}}{3\sqrt{3\pi}}\frac{1}{C_0}$, there is a function $f\in C(\TT)$, such that $\|T_{\Epsilon}[f]\mathbbm{1}_E\|_{L_2(\TT)}\geq c$.
\end{proof}

\begin{remark}
By Remark \ref{REMC0}, we have that $\frac{2}{3\sqrt{3\pi}}\leq \frac{2\sqrt{2}}{3\sqrt{3\pi}}\frac{1}{C_0}\leq\frac{2\sqrt{2}}{3\sqrt{3\pi}}$.  The left hand side is approximately $0.217157$, while the right hand side is approximately $0.307106$.  Thus, it may be that the constant $C_1$ is rather small.  However, the method of proof used here is somewhat arbitrary in that we restricted our considerations to functions $f\in C(\TT)$ which are trigonometric polynomials with coefficients $\pm1$.  To get the best bound on $C_1$, one should consider more general functions.

In particular, we find it an interesting problem to consider if it is possible that $C_1=1$.
\end{remark}

The following corollary is immediate from Theorem \ref{THMC0C1}.

\begin{corollary}\label{CORc}
There exists a $c>0$ (in particular, $c=2\sqrt2/(15\sqrt{3\pi})$ works) such that for every $E\subset\TT$ with $|E|>0$, there exists an $f\in C(\TT)$ with $\|f\|_{\infty}\leq1$ and $\Epsilon\in\{-1,1\}^\Z$ such that $\|T_{\Epsilon}[f]\mathbbm{1}_E\|_{2}\geq c$.
\end{corollary}

\begin{proof}[Proof of Theorem \ref{THMVLinfty}]
Combine Corollary \ref{CORc} with Proposition \ref{equivLinfit}.
\end{proof}

Another question that we think can be of interest is the following.
\begin{problem}
What is the largest subspace on which the Greedy rearrangement of the Fourier series (rearranging by decreasing order of the absolute value of the Fourier coefficients) converges in SOT for every continuous function?
\end{problem}

\section{Characterization of $\W_{S_\infty}$}\label{SECW}

The purpose of this section is to characterize $\W_{S_\infty}$.  Notice that  the Wiener algebra of functions with absolutely summable Fourier coefficients is contained in $\W_{S_\infty}$.  Indeed, let
$$A(\TT):=\left\{f\in L_1(\TT):\sum_{n\in\Z}|\widehat{f}(n)|<\infty\right\}.$$
For $f\in A(\TT)$, $S_N[f]$ converges absolutely (hence unconditionally) to $f$ in the uniform norm, which implies that $A(\TT)\subset \W_{S_\infty}$ since $\|S_{\sigma,N}[f]g-g\|_2\leq\|S_{\sigma,N}[f]-f\|_\infty\|g\|_2$.

However, it turns out that indeed, $\W_{S_\infty}$ is the Wiener algebra.  That is:
\begin{theorem}\label{equivWinf}
$\W_{S_\infty} = A(\TT).$
\end{theorem}


To prove the theorem above, we will utilize two lemmas.  The first is of independent interest.


\begin{lemma}\label{wienerlem}
For $f\in L_\infty(\TT)$, if $\widehat f(n)\geq 0$ for all $n\in\Z$, then $f\in A(\TT)$.
\end{lemma}
\begin{proof} Fix an $N\in\N$, and let $E:=\{-N,\dots,N\}$.  Fix a $\xi\in L_{2},$ and let $\eta\in L_{2}$ be such that $\widehat{\eta}(n)=|\widehat{\xi}(n)|$ for all $n\in \Z.$ Observe that
\[|\widehat{S_{N}[f]\xi}(n)|=|(\mathbbm{1}_{E}\widehat{f})*\widehat{\xi}(n)|\leq \widehat{f}*|\widehat{\xi}|(n)\leq \widehat{f}*\widehat{\eta}(n)=\widehat{f\eta}(n).\]
Hence by Plancherel's Theorem, $\|S_{N}[f]\xi\|_{2}\leq \|f\eta\|_{2}\leq \|f\|_{\infty}\|\eta\|_{2}=\|f\|_{\infty}\|\xi\|_{2}.$ Since $\xi$ was arbitrary, we have that $\|S_{N}[f]\|_\infty\leq \|f\|_{\infty}.$

Then also
\[\sum_{-N}^N\widehat{f}(n)=S_{N}[f](0)\leq\|S_{N}[f]\|_\infty\leq \|f\|_\infty<\infty,\]
which implies that $f\in A(\TT)$.
\end{proof}

\begin{lemma}\label{L:basic facts}
\begin{enumerate}[(i)]
\item   $f\in \W_{S_{\infty}}$ if and only if $\sup_{\sigma\in S_\infty, N\in\N}\|S_{\sigma,N}[f]\|_\infty<\infty.$  \label{I:going to the PUB}
\item For every $f\in \W_{S_{\infty}}$ and $J\subset \Z$ (finite or not), the function $g$ given by $\widehat{g}(n)=\mathbbm{1}_{J}(n)\widehat{f}(n)$ is in $\W_{S_{\infty}}.$    \label{I:back at the PUB}
\end{enumerate}

\end{lemma}

\begin{proof}

(\ref{I:going to the PUB}): First suppose that $f\in \W_{S_{\infty}}.$ Fix a $\xi\in L_{2}.$ By the principle of uniform boundedness, it is enough to show that $\sup_{\sigma,N}\|S_{\sigma,N}[f]\xi\|_{2}<\infty.$  Since $f\in\W_{S_\infty}$, $\sum_{n\in \Z}c_ne_{n}\xi$ converges unconditionally in $L_{2},$ so it follows from standard Banach space theory that indeed
\[\sup_{\sigma\in S_\infty, N\in\N}\|S_{\sigma,N}[f]\xi\|_{2}<\infty.\]

Conversely, suppose that $\sup_{\sigma,N}\|S_{\sigma,N}[f]\|_\infty=C<\infty.$ Let $\sigma\colon \Z\to \Z$ be any bijection. Then for every $N\in\N$, $\|S_{\sigma,N}[f]\|_\infty\leq C,$ hence by the same argument as in the proof of Theorem \ref{THMEquivalences}, $M_{S_{\sigma,N}f}\to M_f$ in the strong operator topology, which implies that $f\in\W_{S_\infty}$.


(\ref{I:back at the PUB}): Let $C=\sup_{\sigma,N}\|S_{\sigma,N}[f]\|<\infty.$  Let $g\in L_2$ be the function given by $\widehat{g}(n)=\mathbbm{1}_{J}(n)\widehat{f}(n).$  Note that for every $\sigma\in S_\infty$ and $N\in\N$, there exists a $\sigma'\in S_\infty$ such that $S_{\sigma,N}[g] = S_{\sigma',N}[f]$, whence
$\|S_{\sigma,N}[g]\|_\infty \leq C$.  Since $\sigma$ and $N$ are arbitrary, the conclusion follows.
\end{proof}

\begin{corollary}\label{C:getting  absolute values}
Given $f\in \W_{S_{\infty}}$ with $\widehat f$ real, define $g\in L_{2}$ by $\widehat{g}=|\widehat{f}|.$ Then $g\in \W_{S_{\infty}}.$
\end{corollary}

\begin{proof}

Let
\[J_{+}=\{n\in \Z:\widehat{f}(n)\geq 0\}, \quad J_{-}=J_+^{c}.\]
By Lemma \ref{L:basic facts}, $f_{+},f_{-}\in \W_{S_{\infty}}$ where
\[\widehat{f_{+}}=\mathbbm{1}_{J_{+}}\widehat{f},\,\,\, \widehat{f_{-}}=\mathbbm{1}_{J_{-}}\widehat{f}.\]
Since $g=f_-+f_+$ so we get $g\in \W_{S_{\infty}}.$
\end{proof}

\begin{proof}[Proof of Theorem \ref{equivWinf}]
We have already noted that $A(\TT)\subset\W_{S_\infty}$.  To see the reverse inclusion, let $f\in \W_{S_\infty}$.  There are $f_{R},f_{I}\in L_{\infty}$ with $\widehat{f_{R}}=\textnormal{Re}(\widehat{f}),\quad \widehat{f_{I}}=\textnormal{Im}(\widehat{f}).$ In fact,
\begin{equation}\label{E:how you get it}
f_{R}(x)=\frac{f(x)+\overline{f(-x)}}{2},f_{I}(x)=\frac{f(x)-\overline{f(-x)}}{2i}.
\end{equation}
Note that  $f\in \W_{S_{\infty}}$ implies that $f_R,f_I \in \W_{S_{\infty}}$. Indeed, fix $\sigma\in S_\infty$ and $N\in\N$ finite. A similar equation as (\ref{E:how you get it}) relates the partial sums of $f_{R},f_{I}$ rearranged according to $\sigma$ to the corresponding rearranged partial sum of $f$. From this, it is easy to see that:
\[\max(\|S_{\sigma,N}[f_{R}]\|,\|S_{\sigma,N}[f_{I}]\|)\leq \|S_{\sigma,N}[f]\|.\]
Thus without loss of generality we can assume that $\widehat{f}$ is real.

Then  by Corollary \ref{C:getting  absolute values}, for $g$ given by $\widehat{g}=|\widehat{f}|$, we have $g\in L_\infty$, and from Lemma \ref{wienerlem}, $g\in A(\TT)$. Note that $\|f\|_{A(\TT)}=\|g\|_{A(\TT)}$, thus $f\in A(\TT)$.

\end{proof}

\section{The WOT Case}\label{SECWOT}

It is natural to consider the analogues of the subspaces $\W_{S_\infty}$ and $\V_{S_\infty}$ when asking only for WOT convergence rather than SOT convergence.  Our remarks here will remain brief, but beginning with the first case, consider $\widetilde{\W}_{S_\infty}$ to be the largest subspace of $C(\TT)$ such that for every $f\in \widetilde{\W}_{S_\infty}$, every $\sigma\in S_\infty$, and every $g,h\in L_2$, we have
\[ \bracket{M_{S_\sigma,N}[f]g,h}\to\bracket{M_fg,h}.\]
That is, $\widetilde{\W}_{S_\infty}$ is the largest subspace of $C(\TT)$ (or equivalently $L_\infty$ as a result of the proof) for which Fourier series converge unconditionally in the WOT on $\mathcal{B}(L_2)$.  Then the following holds.

\begin{proposition}
$\widetilde{\W}_{S_\infty} = \W_{S_\infty} = A(\TT)$.
\end{proposition}

\begin{proof}
Fix $f\in \widetilde{\W}_{S_\infty}.$  For $g\in L_{2}$ and a finite $E\subset \Z,$ we define a linear functional:
$L_{E,g}\colon L_{2}\to \mathbb{C}$
by
\[L_{E,g}(h)=\bracket{h,S_{E}[f]g}.\]

Fix a $g\in L_{2}.$ Since the Fourier series of $f$ converges unconditionally in the weak operator topology, we have for each $h\in L_{2}$ that:
\[\sup_{E}|L_{E,g}(h)|=\sup_{E}|\bracket{S_{E}(f)g,h}|<\infty,\]
where the supremum is over all finite subsets of $\Z^d.$
It follows from the principle of uniform boundedness that there is a $C_{g}\geq 0$ so that:
\[\|L_{E,g}\|\leq C_{g}.\]

However, since Hilbert spaces can be identified (in a conjugate linear, isometric way) with their own dual, we have that
\[\|L_{E,g}\|=\|S_{E}(f)g\|.\]
Therefore, for each $g\in L_{2},$ we have that
\[\sup_{E}\|S_{E}(f)g\|_{2}<\infty.\]
Thus another application of the principle of uniform boundedness shows that $\sup_{E}\|S_{E}(f)\|_{\infty}<\infty.$ Hence we have that $f\in \W_{S_{\infty}},$ which is $A(\TT^d)$ by Theorem  \ref{THMWd}.
\end{proof}

The analogue of $\V_{S_\infty}$ requires a bit more consideration. Given a subspace $E\subseteq L_2(\TT),$ and a collection of operators $T_n$ on $L_2(\TT),$ we say that \emph{$\sum_{n}T_{n}$ converges unconditionally in the weak operator topology on $E$} if for every $v,w\in E$ we have that $\sum_{n}\la {T_{n}v,w\ra}$ converges unconditionally. This is a mild abuse of terminology, since the $T_n$ are not assumed to leave $E$ invariant, but that is not important for our purposes. It follows from Zorn's Lemma that there exists a maximal subspace of $L_2$ for which the Fourier series of every continuous function converges unconditionally in the weak operator topology on that subspace.  However, it is not clear that this subspace is unique. For example, if every Fourier series converges unconditionally in the weak operator topology on $E,F,$ there is, a priori, no reason to expect that every Fourier series converges unconditionally in the weak operator topology on $E+F.$ We will, in fact, show that there is a unique maximal subspace of $L_2$ so that the Fourier series of every continuous function converges unconditionally in the weak operator topology on that subspace, and that this subspace is $L_4(\TT),$ but the reader should note that even uniqueness is certainly not clear.

 To fix this issue, we define a subspace of $L_{1}.$ We let $\widetilde{\V}$ be the set of functions $g\in L_{1}(\TT),$ so that for every $f\in L_\infty(\TT)$ we have that  $\sum_{n\in \Z}\widehat{f}(n)\la e_{n}, g\ra=\sum_{n\in \Z}\widehat{f}(n)\overline{\widehat{g}(n)}$ converges unconditionally. Given a subspace $E$ of $L_2,$ we have that the Fourier series converges unconditionally in the weak operator topology restricted to $E$ if and only if $g=\overline{h}k\in \widetilde{\V}$ for all $h,k\in E.$ We abuse notation here, to mean $\la e_nh,k\ra=\la e_n,\bar hk\ra=\la e_n, g\ra$ and now $g\in L_1$.  Hence $\widetilde{\V}$ is a substitution for a unique maximal subspace on which the Fourier series converges unconditionally in the weak operator topology.


\begin{lemma}\label{L:basic properties of V}
We have the following properties of $\widetilde{\V}.$
\begin{enumerate}[(i)]
\item  A function $g\in L_{1}$ is in $\widetilde{\V}$ if and only if $f*g\in A(\TT)$ for every  $f\in L_\infty(\TT).$\label{I:convolution for V in WOT}
\item If $g\in \widetilde{\V},$ then the map $C_{g}\colon L_\infty(\TT)\to A(\TT)$ given by $C_{g}(f)=f*g$ is a well-defined bounded operator, i.e. $g$ is a bounded Fourier multiplier from $L_\infty$ to $A(\TT)$. \label{I:bounded operator for V in WOT}
\item A function $g\in L_{1}$ is in $\widetilde{\V}$ if and only if only  $\sup_{E}|\la S_{E}(f),g\ra|<\infty$ for all $f\in L_\infty(\TT),$ where the supremum is over all finite subsets $E$ of $\Z.$ \label{I:boundedness of inner products L1}
\end{enumerate}
\end{lemma}




\begin{proof}

(\ref{I:convolution for V in WOT}):
Note that the scalar series $\sum_{n}\widehat{f}(n)\overline{\widehat{g}(n)}$  converges unconditionally if and only if
\[\sum_{n}|\widehat{f}(n)||\overline{\widehat{g}(n)}|=\sum_{n}|\widehat{f}(n)||\widehat{g}(n)|=\sum_{n}|(\widehat{f*g})(n)|<\infty,\]
which is equivalent to $f*g\in A(\TT)$.



(\ref{I:bounded operator for V in WOT}):
By (\ref{I:convolution for V in WOT}), the linear map $C_{g}$ is well-defined when restricted to $L_\infty(\TT)$. We first show that it is bounded restricted to $L_\infty(\TT)$ by applying the closed graph theorem. Suppose that $f_{n},f$ are in $L_\infty(\TT)$ with $f_{n}\to f$ in norm, and that $h\in A(\TT)$ and $C_{g}(f_{n})\to h$ in $A(\TT).$ We then have for every $m\in \Z$ that
\[\widehat{h}(m)=\lim_{n}\widehat{C_{g}(f_{n})}(m)=\lim_{n}\widehat{f_{n}}(m)\widehat{g}(m)=\widehat{f}(m)\widehat{g}(m).\]
Hence $h=C_{g}(f),$ so $C_{g}$ has a closed graph and thus a bounded operator when restricted to $L_\infty(\TT).$

Now let $f\in L_\infty,$ and fix a sequence $f_{k}\in C(\TT)$ with $\|f_{k}\|\leq \|f\|_{\infty}$ and $\|f_{k}-f\|_{2}\to 0$ (e.g. use the Fej\'{e}r means of $f$). Then, by Fatou's lemma
\[\|\widehat{f}\widehat{g}\|_{1}\leq \liminf_{k\to\infty}\|\widehat{f_{k}}\widehat{g}\|_{1}=\liminf_{k\to\infty}\|C_{g}(f_{k})\|_{A(\TT)}\leq \|C_{g}\|_{C(\TT)\to A(\TT)}\|f\|_{\infty}.\]



(\ref{I:boundedness of inner products L1}):
For $f\in L_{\infty}(\TT)$
\[\la S_{E}(f), g\ra =\sum_{n\in E}\widehat{f}(n)\overline{\widehat{g}(n)}.\]
Then we use the basic fact from real analysis that
\[\sum_{n\in \Z}|(\widehat{f*g})(n)|=\sum_{n\in \Z}|\widehat{f}(n)\widehat{g}(n)|=\sum_{n\in \Z}|\widehat{f}(n)\overline{\widehat{g}(n)}|<\infty\]
if and only if
\[\sup_{E}|\sum_{n\in E}\widehat{f}(n)\overline{\widehat{g}(n)}|<\infty.\]



\end{proof}

We now proceed to show that $\widetilde{\V}=L_{2}.$ For a locally compact Hausdorff space $X,$ we use $\Prob(X)$ for the space of all Radon probability measures on $X.$ We use $M(X)$ for the Banach space of all Radon measures on $X.$ If $X$ is discrete, we identify $M(X)$ with $\ell_{1}(X)$ in the natural way.

Before proving that $\widetilde{\V}=L_{2}(\TT),$ it may be helpful to explain some of the ideas behind the proof. By the above, every $g\in \widetilde{\V}$ gives rise to a bounded operator $T\colon L_{\infty}(\TT)\to \ell_{1}(\Z)$ given by $T(f)=\widehat{f}\widehat{g} $. If we instead knew that we had a bounded $L_2(\TT)\to \ell_{1}(\Z)$ operator given by $f\mapsto\widehat{f}\widehat{g},$ then we would automatically know that $\widehat{g}\in \ell_2(\Z).$ This is because $\ell_2(\Z)\widehat{g}\subseteq \ell_1(\Z),$ and it is well known that this implies that $\widehat{g}\in \ell_{2}(Z)$ and thus $g\in L_2(\TT)$ by Parseval's theorem. So the goal now is to try to ``upgrade'' our given operator $L_{\infty}(\TT)\to \ell_{1}(\Z)$ to an operator $L_2(\TT)\to \ell_{1}(\Z).$

In fact, we can rephrase the logic of the above paragraph by saying that if $T$ factors through the inclusion $\iota\colon L_{\infty}(\TT)\to L_{2}(\TT),$ i.e. $T=S\circ \iota$ for some bounded operator $S\colon L_{2}(\TT)\to \ell_{1}(\Z),$ then $g\in L_2(\TT).$ So one approach would be to try to prove that $T$ factors through this inclusion map. We could not prove this directly, but we can apply the Grothendieck theorem (see \cite[Theorem 5.5]{PisierFact}) to show a slightly weaker result: there is an $\eta\in \Prob(\Z)$ so that $T\big|_{C(\TT)}$ factors through $j_{\mu}\colon \ell_{2}(\Z,\eta)\to \ell_1(\Z)$ given by $j_{\mu}(f)=f\eta.$ This ends up being sufficient to prove that $\widetilde{\V}=L_2(\TT).$ For ease of reading, let us state precisely what version of the Grothendieck theorem we are using.

\begin{theorem}[Theorem 5.5 of \cite{PisierFact}]\label{T:GT restated}
Let $X,Y$ be locally compact Hausdorff spaces, and let $T\colon C(X)\to M(Y)$ be a bounded linear map. Then there are probability measures $\mu\in\Prob(X),\eta\in \Prob(Y),$ and a bounded linear map $U\colon L^{2}(X,\mu)\to L^{2}(Y,\eta)$ so that $T=\iota_{\eta}^{t}\circ U\circ \iota_{\mu}$ where:
\begin{itemize}
\item $\iota_{\mu}\colon C(X)\to L^{2}(X,\mu)$, $\iota_{\eta}\colon C(Y)\to L^{2}(Y,\nu)$ are the canonical maps sending a continuous function to its $L^{2}$-class,
\item $^{t}$ denotes the Banach space transpose.
\end{itemize}
\end{theorem}
Theorem 5.5 of \cite{PisierFact} only proves this when $X,Y$ are compact, however the remarks after Theorem 5.5 of \cite{PisierFact} as well as \cite[Corollary 5.8]{PisierFact} imply the above version of Grothendieck's theorem.

We will actually only need Theorem \ref{T:GT restated} in the case that $X$ is compact, and $Y$ is countably infinite and discrete. In this case, we can find a bijection $\phi\colon Y\to K$ where $K$ is (any) countably infinite, compact Hausdorff space. Set $\widetilde{T}\colon C(X)\to M(K)$ by $\widetilde{T}(f)=\phi_{*}(T(f)).$ Let $\widetilde{U},\mu\in \Prob(X),\widetilde{\nu}\in \Prob(K)$ be as in Theorem \ref{T:GT restated} for $\widetilde{T}.$ Define $\eta\in \Prob(Y)$ by $\eta=(\phi^{-1})_{*}(\widetilde{\eta})$ and $U\colon L^{2}(X,\mu)\to L^{2}(Y,\eta)$ by $U(f)=\widetilde{U}(f)\circ \phi.$ It is direct to check that $U,\mu,\eta$ now satisfy the Grothendieck theorem for $T.$ This proof can be generalized: given any  locally compact, metrizable space $Y,$ there is a compact, metrizable space $K$ and a Borel bijection $\phi\colon Y\to K$ with Borel inverse (see \cite[Theorem 3.3.13]{Sriva}). If $Y$ is additionally assumed metrizable in Theorem \ref{T:GT restated}, then this reduces the case that $Y$ is locally compact in \ref{T:GT restated} to the case that $Y$ is compact Hausdorff.

\begin{lemma}
Suppose that $g\in \widetilde{\V},$ then there are $\mu \in \Prob(\TT),\eta\in \Prob(\Z)$  and a bounded linear operator $U\colon L_{2}(\TT,\mu)\to \ell_{2}(\Z,\eta)$ so that $\widehat{f} \widehat{g}=U(f)\eta$ for all $f\in L_\infty(\TT)$.

\end{lemma}

\begin{proof}
Let $T\colon L_\infty(\TT)\to \ell_{1}(\Z)$ be given by $T(f)=\widehat{C_{g}(f)};$ by the previous lemma we know that $T$ is well defined and bounded. Recall that  there is an obvious isometric  isomorphism $ \ell_{1}(\Z)\cong M(\Z)^{*}$. So we can find $\mu\in \Prob(\TT)$, $\eta\in \Prob(\Z),$ and $U$ as in the Grothendieck theorem (Theorem \ref{T:GT restated}) for this map $T.$

 It is easy to see that $\iota_{\mu}^{t}(f)=f\eta$ for all $f\in \ell_2(\Z,\eta).$
We then have, for all $f\in C(\TT):$
\[\widehat{f}\widehat{g}=T(f)=\iota_{\mu}^{t}U(f)=U(f)\eta\]
as desired.

\end{proof}

\begin{theorem}\label{T:V for L1}
We have that $\widetilde{\V}=L_2(\TT).$
\end{theorem}

\begin{proof}
Let the operator $U$ and the measures $\mu,\eta$ be as in the preceding lemma. Define $S\colon C(\TT) \to L_2(\Z,\eta)$ by $S(f)=U(\iota_{\mu}(f))$ where $\iota_{\mu}\colon C(\TT)\to L_2(\TT,\mu)$ sends a continuous function to its $L_2(\mu)$-class. Then $\widehat{f}\widehat{g}=S(f)\eta$ for all $f\in L_\infty(\TT).$ Apply this to $f=e_{n}$ to see that  $\widehat{g}(n)=S(e_n)(n)\eta(n)$ and so:
\[|\widehat{g}(n)|^{2}= |S(e_{n})(n)\eta(n)|^{2}=|S(e_n)(n)|^{2}\eta(n)\eta(n)\leq \|S(e_{n})\|_{\ell_{2}(\eta)}^2\eta(n)\leq \|S\|^2\eta(n).\]
So $|\widehat{g}|^{2}\leq \|S\|\eta\in \ell_{1}(\Z).$ Hence $\widehat{g}\in \ell_{2}(\Z),$ so by Parseval's theorem we know that $g\in L_2(\TT).$
\end{proof}

\begin{corollary}
\begin{enumerate}[(i)]
\item If $f\in L_\infty(\TT),$ then the Fourier series for $f$ converges unconditionally in the weak operator topology on $L_4(\TT).$ \label{I:easy part unconditional WOT}
\item Conversely, if $E$ is a subspace of $L_2(\TT),$ and the Fourier series for $f$ converges unconditionally in the weak operator topology on $E$ for every $f\in L^{\infty}(\TT),$ then $E\subseteq L_4(\TT).$ \label{I:hard part unconditional WOT}
\item  Statements $(i)$ and $(ii)$ remain valid if $L_\infty(\TT)$ is replaced by $C(\TT)$.

\end{enumerate}
\end{corollary}

\begin{proof}
(\ref{I:easy part unconditional WOT}): Suppose that $g,h\in L_4(\TT),$ and fix an $f\in L_\infty(\TT).$ Given $\varepsilon>0,$ choose a finite $E\subseteq \Z$ so that $\left\|\sum_{n\in \Z\setminus E}\widehat{f}(n)e_{n}\right\|_{2}<\varepsilon.$ Then:
\begin{align*}
\left|\left\la\left(f-\sum_{n\in E}\widehat{f}(n)e_n\right)g,h\right\ra\right|&=\left|\left\la\left(f-\sum_{n\in E}\widehat{f}(n)e_{n}\right),\overline{g}h\right\ra\right|\\
&=\left|\left\la\left(\sum_{n\in \Z\setminus E}\widehat{f}(n)e_{n}\right),\overline{g}h\right\ra\right|\\
&\leq \left\|\sum_{n\in \Z\setminus E}\widehat{f}(n)e_{n}\right\|\|gh\|_{2}\\
&\leq \varepsilon\|g\|_{4}\|h\|_{4}.
\end{align*}
Since $g,h$ are fixed and $\varepsilon>0$ is arbitrary, this proves that the Fourier series of $f$ converges unconditionally in the weak operator topology on $L_4(\TT).$

(\ref{I:hard part unconditional WOT}): Assume that $E\subseteq L_2(\TT)$ and that the Fourier series for $f$ converges unconditionally in the weak operator topology on $E.$ It is easy to see that for every $h,k\in E,$ we have that $h\overline{k}\in \widetilde{\V},$ where $\widetilde{\V}$ is defined as in the discussion before Lemma \ref{L:basic properties of V}. By Theorem \ref{T:V for L1}, we have that $h\overline{k}\in L_2(\TT)$ for every $h,k\in \widetilde{\V}.$

Now fix a $g\in \widetilde{\V}.$ Applying the discussion of the preceding paragraph to $h=g=k,$ we see that $|h|^{2}\in L_2(\TT).$ Hence $h\in L_4(\TT).$

\end{proof}


\section{The Multivariate Case}\label{SECMulti}

In this subsection, we briefly discuss the extensions of the previous results to higher dimensions.  Essentially all results hold without change for arbitrary dimension once the statements are interpreted correctly.  To begin, let us state the analogue of \Rev 's theorem (Theorem \ref{THMRevesz}) for multivariate case.

For higher dimensions, we replace the notion of a subsequence of rearranged partial sums with partial sums with respect to a nested exhaustion of $\Z^d$.  Precisely, for $f\in C(\TT)$ and a finite subset $J$ of $\Z^d$ define
\[S_{J}[f](t):=\sum_{n\in J}c_ne^{2\pi int},\quad t\in\TT\]
where
\[ c_n:=\widehat{f}(n):=\dint_{\TT^d} f(t)e^{-2\pi i\langle n,t\rangle}dt,\;\quad n\in \Z^d.\]

The following is \Rev's analogue of Theorem \ref{THMRevesz} for arbitrary dimension; the argument is essentially the same as the one-dimensional case.

\begin{theorem}[\Rev, \cite{revesz1990}]\label{T:Revesz Multivariate} For every $f\in C(\TT^d)$, there exists a nested sequence of finite sets $\{J_k\}\subset\Z^d$ with $\bigcup_{k}J_k=\Z^d$ such that
$\|S_{J_k}[f]-f\|_\infty\to 0$ as $k\to \infty$.
\end{theorem}

Since \Rev's Theorem holds in arbitrary dimension, the arguments of Theorem \ref{THMEquivalences} remain valid in arbitrary dimension as well once the notion of a rearrangement is changed to a nested exhaustion of $\Z^d$ with finite sets as in Theorem \ref{T:Revesz Multivariate}.

\subsection{Characterization of $\mathcal{V}_{S_\infty}^{(d)}$ and $\mathcal{W}_{S_\infty}^{(d)}$}

In this subsection, we provide the analogue of the characterization of $\V_{S_\infty}$ and $\W_{S_\infty}$ for $\Z^d$.  Namely, if $\V_{S_\infty}^{(d)}$ is the largest subspace of $L_2(\TT^d)$ on which the multiplication operators $M_{S_{J_k}[f]}\to M_f$ in the SOT on $\V_{S_\infty}^{(d)}$ for every nested exhaustion of $\Z^d$ by finite sets $\{J_k\}$ and every $f\in L_\infty(\TT^d)$, then we show that $\V_{S_\infty}^{(d)}=L_\infty(\TT^d)$.  Likewise, if $\W_{S_\infty}^{(d)}$ and $\widetilde{\W}_{S_\infty}^{(d)}$ are the largest subspaces of $L_\infty(\TT^d)$ for which $M_{S_{J_k}[f]}\to M_f$ in the SOT, respectively WOT, on $L_2$ for every exhaustion of $\Z^d$, then $\W_{S_\infty}^{(d)}=\widetilde{\W}_{S_\infty}^{(d)}=A(\TT^d)$, the $d$--dimensional Wiener algebra.

To begin, we note the following direct analogue of Proposition \ref{equivLinfit}.  As for notation, given a subset $J$ of $\Z^d$ and $\Epsilon\in\{-1,1\}^{\Z^d}$, define
\[T_{\Epsilon,J}[f](t):=\sum_{n\in J}\epsilon_n c_ne^{2\pi i\bracket{n,t}}.\]  Likewise, $T_\Epsilon[f]$ is the $L_2$ limit of $T_{\Epsilon,J_k}[f]$ where $\{J_k\}$ is any nested exhaustion of $\Z^d$ via finite subsets.  The proof of the following is the same as the univariate case, so is omitted.

\begin{proposition}\label{PROPVMultivariate}
The following are equivalent:
\begin{enumerate}[(i)]
\item $\V_{S_\infty}^{(d)}=L_\infty(\TT^d)$.
\item There exists an absolute constant $c>0$ such that for any measurable set $E\subset\TT^d$ with $|E|>0$ there exists a function $f\in C(\TT^d)$ with $\|f\|_\infty\leq1$, and there exists $\Epsilon\in\{-1,1\}^{\Z^d}$ for which
\[\|T_{\Epsilon}[f]\mathbbm{1}_E\|_2\geq c.\]
\end{enumerate}
\end{proposition}

Our estimate for the constants involved in Theorem \ref{THMC0C1} are dimension-dependent; however the estimate still gives a positive answer as follows.

\begin{theorem}
If $\tilde{C}_1$ is the largest constant for which condition \textit{(ii)} of Proposition \ref{PROPVMultivariate} holds for all $0<c<\tilde{C}_1$, and $C_0$ is the constant from Problem \ref{PROBC0}, then the following holds:
\[ \tilde{C}_1 \geq \left(\dfrac23\right)^\frac{3d}{2}\pi^{-\frac{d}{2}}\dfrac{1}{C_0^d}.\]
Consequently, $\V_{S_\infty}^{(d)}=L_\infty(\TT^d)$ for every $d\in\N$.
\end{theorem}

\begin{proof}
Let $0<\gamma_0<1$ be arbitrary, and let $E\subset\TT^d$ have positive Lebesgue measure.  Without loss of generality, $0$ is a point of Lebesgue density of $E$, from which it follows that there exists some $\gamma_1>0$ such that for every $0<\eps<\gamma_1$,
\[1 \geq |E\cap[-\eps,\eps]^d|\geq(2\eps)^d(1-\gamma_0)^2.\]
Define the tensor version of the sinc function via
\[\sinc^{\otimes d}(t) = \sinc(t_1)\dots\sinc(t_d),\quad t\in\R^d,\] and via \eqref{EQsinc}, we have that if $0<\delta<1$ and $\|t\|_\infty\leq\delta$, then
\[|\sinc^{\otimes d}(t)|\geq \prod_{j=1}^d(1-|t_j|)\geq(1-\delta)^d.\]
Now given $\delta$, choose an $N\in\N$ such that $\frac{\delta}{\pi(N+1)}<\gamma_1$, and let $(\eps_n)_{n\in\N}\subset\{-1,1\}$ be a sequence satisfying the conclusion of Problem \ref{PROBC0} with constant $C_0$, and let $f(t) = \frac{1}{C_0\sqrt{N+1}}\finsum{n}{0}{N}\eps_n e^{2\pi int}$ be the function in the proof of Theorem \ref{THMC0C1}.  Next, set $F = f^{\otimes d}$, and notice that $\|F\|_{L_\infty(\TT^d)} \leq 1$ since $\|f\|_{L_\infty(\TT)}\leq1$.
Define \[\Lambda:=\{n\in\Z^d:0\leq n_i\leq N, i=1,\dots, d\}.\]  Then let $\tilde{\Epsilon}$ be defined by $\tilde{\eps}_n=\eps_{n_1}\dots\eps_{n_d}$ for $n\in\Lambda$.  Thus, by the same calculation as in the proof of Theorem \ref{THMC0C1} we have
\begin{align*}
 T_{\tilde{\Epsilon}}[f](t) & = \dfrac{1}{C_0^d (N+1)^\frac{d}{2}}\sum_{n\in\Lambda}e^{2\pi i\bracket{n,t}}\\
 & = \dfrac{1}{C_0^d (N+1)^\frac{d}{2}}\left(\finsum{n}{0}{N}e^{2\pi int}\right)^{\otimes d}\\
 & = \dfrac{1}{C_0^d (N+1)^\frac{d}{2}}(-1)^{Nd}\left(\dfrac{\sin(\pi(N+1)t)}{\sin(\pi t)}\right)^{\otimes d}.\\
\end{align*}
Thus, if $\|t\|_\infty\leq\frac{\delta}{\pi(N+1)}$, we have that \[|T_{\tilde{\Epsilon}}[f](t)|\geq\dfrac{1}{C_0^d (N+1)^\frac{d}{2}}(1-\delta)^d(N+1)^d = \dfrac{(N+1)^\frac{d}{2}}{C_0^d}(1-\delta)^d.\]
Finally,
\begin{align*}
\|T_{\tilde{\Epsilon}}[f]\mathbbm{1}_E\|_2 & \geq \|T_{\tilde\Epsilon}[f]\mathbbm{1}_{E\cap[-\frac{\delta}{\pi(N+1)},\frac{\delta}{\pi(N+1)}]}\|_2\\
& \geq \dfrac{(N+1)^\frac{d}{2}}{C_0^d}(1-\delta)^d \left|E\cap\left[-\frac{\delta}{\pi(N+1)},\frac{\delta}{\pi(N+1)}\right]\right|^\frac12\\
& \geq \dfrac{(N+1)^\frac{d}{2}}{C_0^d}(1-\delta)^d\left(\dfrac{2\delta}{\pi(N+1)}\right)^\frac{d}{2}(1-\gamma_0).\\
\end{align*}
Maximizing over $\delta$, and choosing $\delta=\frac13$, we find that (as $\gamma_0$ was arbitrary), \[\tilde{C}_1\geq \dfrac{2^\frac{3d}{2}}{3^\frac{3d}{2}\pi^\frac{d}{2}}\dfrac{1}{C_0^d},\] which completes the proof.
\end{proof}

It should be noted that the proof above gives a constant $\tilde{C}_1$ which degrades quickly as the dimension increases.  It would be interesting to determine if this can be avoided, i.e. if a dimensionless lower bound could be obtained.

 Additionally, the proof of Theorem \ref{T:V for L1} holds for any dimension, and so we have the following.

\begin{theorem}\label{T:VWOT Multivariate}
If $f\in L_\infty(\TT^d)$, then the Fourier series of $f$ converges unconditionally in the weak operator topology on $L_4(\TT^d)$.  Conversely, if $E$ is a subspace of $L_2(\TT^d)$, and the Fourier series of $f$ converges unconditionally in the weak operator topology on $E$ for every $f\in L_\infty(\TT^d)$, then $E\subseteq L_4(\TT^d)$.  Moreover, both statements hold with $L_\infty(\TT^d)$ replaced by $C(\TT^d)$.
\end{theorem}

Next, the following theorem holds by the same proof as in Section \ref{SECW}.   One needs only notice that the proof of Lemma \ref{L:basic facts} is the same line by line in any dimension.

\begin{theorem}\label{THMWd}
For every $d\in\N$, $\W_{S_\infty}^{(d)}=\widetilde{\W}_{S_\infty}^{(d)}=A(\TT^d)$, where $A(\TT^d)$ is the Wiener algebra of functions on $\TT^d$ with absolutely convergent Fourier series.
\end{theorem}

\section{Conclusion}\label{SECConclusion}

We have considered Ulyanov's problem concerning uniform convergence of rearranged Fourier series and, inspired by \Rev 's partial results in that direction, we have given some equivalent forms of the problem in terms of SOT or WOT convergence of rearranged Fourier series which are, {\em a priori}, weaker conditions.  From this new operator theoretic perspective, we have given some guidelines for possible ways to find counterexamples to the problem and ideas of how to exhibit a proof thereof.  More specifically, we considered reductions of the problem to convergence on subspaces of $\mathcal{B}(L_2)$ and the relation between subsets of $C(\TT)$, $L_2$, and $S_\infty$, and the behavior of Fourier series related to such triples. This provides many interesting open questions for further research, and also gives a roadmap for future work on the problem.

In summary, we have expressed the following ideas:


\underline{Strategies to prove Ulyanov's problem:}

\begin{itemize}
\item Better understand the structure of $\W_{\{\sigma\}}$, or particularly $\W_{\{id\}}$.
\item Show that $\underset{\sigma\in S_\infty}\bigcup\W_{\{\sigma\}}=C(\TT)$.
\item Prove one of the equivalent conditions in Theorem \ref{THMEquivalences}.
\end{itemize}

\underline{Questions of independent interest:}

\begin{itemize}
\item Determine the optimal constants $C_0$ and $C_1$.
\item Find a dimension-independent bound for $\widetilde{C}_1$ (if such a bound exists).
\item Determine the behavior of the {\em Greedy rearrangements} of Fourier series as multiplication operators.
\end{itemize}

In addition to these, we characterized the largest subspace of $L_\infty$ functions for which the Fourier series converge unconditionally in the SOT and WOT, namely the Wiener algebra, $A(\TT^d)$.  Dual to this, we characterized the largest subspace of $L_2$ on which the Fourier series of all continuous functions converges unconditionally in the SOT restricted to that subspace, and this turned out to be $L_\infty(\TT^d)$.  While the WOT analogue of the last result is not altogether obvious, we gave one way of classifying the analogous WOT subspace, and showed that the unique maximal subspace giving unconditional WOT convergence is $L_4(\TT^d)$.

\section*{Acknowledgements}
The authors thank Szil\'{a}rd \Rev, Rudy Rodsphon, Alex Powell, Alexander Olevskii and Artur Sahakian for valuable discussions related to this work. We also thank the anonymous referees for valuable suggestions. 

Part of this work was done while the first author was an Assistant Professor at Vanderbilt University.  The first author acknowledges partial support in the later stages from the National Science Foundation TRIPODS program, grant CCF--1740858.

Part of this work was done while the second author was an Assistant Professor at Vanderbilt University. The second author acknowledges support from the grants NSF DMS-1600802 and  NSF DMS-1827376.

Part of the work of this paper was done while the third author was a Ph.D. student at Vanderbilt University; the project was motivated by a question asked to him by Vaughan Jones during a class.  The third author acknowledges support during the final stages of the project  by the Oak Ridge National Laboratory, which is operated by UT-Battelle, LLC., for the U.S. Department of Energy under Contract DE-AC05-00OR22725.

%

\end{document}